%% file: paper.tex
\documentclass[12pt,reqno]{amsart}
\usepackage{amssymb}
\usepackage{epsfig}
\usepackage{url}
\usepackage{color}

\newcommand{\C}{\mathbb C}

\newcommand{\R}{\mathbb R}

\newcommand{\N}{\mathbb N}

\newcommand{\F}{\mathcal F}
\newcommand{\HH}{\mathcal H}
\newcommand{\TP}{\mathbb {TP}}

\DeclareMathOperator{\type}{type}

\newtheorem{theorem}{Theorem}[section]
\newtheorem{proposition}[theorem]{Proposition}
\newtheorem{corollary}[theorem]{Corollary}
\newtheorem{question}[theorem]{Question}
\newtheorem{conjecture}[theorem]{Conjecture}

\newtheorem{lemma}[theorem]{Lemma}
\newtheorem{algorithm}[theorem]{Algorithm}

\theoremstyle{definition}
\newtheorem{definition}[theorem]{Definition}
\newtheorem{example}[theorem]{Example}
\theoremstyle{remark}
\newtheorem{remark}[theorem]{Remark}

\title[The $4\times 4$ minors of a $5\times n$ matrix are a tropical basis]{The $4\times 4$ minors of a $5\times n$ matrix\\ are a tropical basis}
\author{Melody Chan}
\address{Department of Mathematics, University of California, Berkeley, 970 Evans Hall, CA 94720-3840, USA}
\author{Anders Jensen}
\address{Mathematisches Institut, Georg-August-Universit\"at G\"ottingen, Bunsen\-stra\ss e 3-5, D-37073 G\"ottingen, Germany}
\author{Elena Rubei}
\address{Dipartimento di Matematica "U.Dini" viale Morgagni 67/A, 50134 Firenze, Italy}
\email{rubei@math.unifi.it}
\thanks{The first author was supported by the Department of Defense (DoD) through the  National Defense Science \& Engineering Graduate Fellowship (NDSEG) Program.}
\thanks{The second author was supported by the German Research Foundation (Deutsche Forschungsgemeinschaft (DFG)) through the Institutional Strategy of the University of G\"ottingen.}
\date{12/28/2009}

\begin{document}

\begin{abstract}
We compute the space of $5\times 5$~matrices of tropical rank at most
$3$ and show that it coincides with the space of $5\times 5$ matrices
of Kapranov rank at most $3$, that is, the space of five labeled
coplanar points in $\TP^4$. We then prove that the Kapranov rank of
every $5\times n$~matrix equals its tropical rank; equivalently, that the
$4\times 4$~minors of a $5\times n$~matrix of variables form a
tropical basis. This answers a question asked by Develin, Santos, and
Sturmfels.
\end{abstract}

\maketitle

\section{Introduction}
The tropical semi-ring $(\R,\oplus,\odot)$,  consisting of the real numbers equip\-ped with tropical addition and multiplication
$$x\oplus y:= \textup{min}(x,y)~~~~~~~~~~\textup{      and      }~~~~~~~~~~x\odot y:=x+y~~~~~~~~~~\textup{   for all   }~~~~~~~~~~x,y\in\R,$$
gives rise to three distinct notions of rank of a tropical matrix
$A\in\R^{d\times n}$. These, tropical rank, Kapranov rank, and Barvinok rank, were studied in \cite{DevelinSantosSturmfels2003}. They arise as the tropicalizations of three equivalent characterizations of matrix rank in the usual sense.

Indeed, classically, a $d\times n$ matrix with entries in a field $K$ has rank
at most $r$ if and only if all of its $(r+1)\times (r+1)$
submatrices
are singular. Equivalently, the set of $d\times n$ matrices of rank at
most $r$ is the determinantal variety defined by the ideal
$J_r^{dn}\subseteq K[x_{11},\dots,x_{dn}]$ generated by the
$(r+1)\times (r+1)$ minors of a $d\times n$-matrix of variables. Finally, this
algebraic variety is the image of the matrix product map
$\phi:K^{d\times r}\times K^{r\times n} \rightarrow K^{d\times n}$.

Accordingly,
the set of matrices of tropical rank
$\leq r$ is defined to be the intersection of the
tropical hypersurfaces defined by the $(r+1)\times (r+1)$ minors in
$K[x_{11},\dots,x_{dn}]$. The set of matrices of Kapranov rank $\leq
r$ is defined to be the tropical variety $T(J_r^{dn})$, while the set of
matrices of Barvinok rank $\leq r$ is the image of the tropicalization
of $\phi$. We will revisit these definitions in Section~\ref{sec:def}.
We note that $T(J_r^{dn})$ can be regarded as the space of $n$ labeled points in $\TP^{d-1}$ for which there exist a tropicalized $r-1$ plane containing them.

Since the intersection of the tropical hypersurfaces defined by a set
of polynomials does not always equal the tropical variety of the ideal they generate, we do not expect Kapranov rank and tropical rank to be
the same.
Similarly, the tropicalization of the image of a polynomial
function is not always equal to the image of its tropicalization; therefore, we do
not expect Barvinok rank and Kapranov rank to be the
same. However, in both of these cases one containment is true, implying
\begin{equation}
\label{eq1}
\textup{Tropical rank}(A) \leq \textup{Kapranov rank}(A) \leq \textup{Barvinok rank}(A).
\end{equation}
as shown in \cite[Theorem~1.4]{DevelinSantosSturmfels2003}.

We are
interested in studying Kapranov rank and tropical rank. The question
of whether these coincide is really a question about tropical bases.
Recall that a tropical basis for an ideal $I$ is a
finite generating set with hypersurface intersection equal to $T(I)$.
The authors of \cite{ctv} prove that any ideal $I$ generated by polynomials in $\C[x_1,\dots,x_n]$ has a tropical basis. It is of fundamental interest to understand the geometry of intersections of tropical hypersurfaces and varieties, and to develop methods to recognize tropical bases.

Using the language of tropical bases, it is natural to ask:
\begin{question}
For which numbers $d,n,$ and $r$ do the $(r+1)\times(r+1)$-minors of a $d\times n$ matrix form a tropical basis? Equivalently, for which $d,n,r$ 
does every $d\times n$ matrix of tropical rank at most $r$ have Kapranov rank at most $r$?
\end{question}

As a corollary to the following theorem, we get that for $d\times n$ matrices with $d$ or $n\leq 4$, the tropical rank and Kapranov rank are equal.
\begin{theorem}\cite[Theorem~5.5, 6.5]{DevelinSantosSturmfels2003}
\label{old theorem}
Let $A\in\R^{d\times n}$. If the tropical rank or the Kapranov rank of $A$ is $1$, $2$, or $\textup{min}(d,n)$, then they are equal.
\end{theorem} 

On the other hand, there exists a $7\times 7$
matrix with tropical rank 3 and Kapranov rank 4 (\cite[Corollary~7.4]{DevelinSantosSturmfels2003}). This matrix is obtained as the cocircuit incidence matrix of the Fano matroid $F_7$; the fact that the Kapranov rank and tropical rank of this matrix differ follows from the non-representability of $F_7$ over a field of characteristic $0$. This is the smallest known example of a matrix whose tropical rank and Kapranov rank are different and shows that the set of $4\times 4$ minors of a $7\times 7$ matrix do not form a tropical basis.

In this paper we answer the following question, asked explicitly by Develin, Santos, and Sturmfels, in the affirmative.
\begin{question}\cite[Section~8,(6)]{DevelinSantosSturmfels2003}
\label{quest:5x5}
Do the $4\times 4$ minors of a $5\times 5$ matrix form a tropical basis?
\end{question}

Our paper is organized as follows. In Section~\ref{sec:5x5}, we
compute the set of $5\times 5$ matrices of tropical rank at most $3$,
regarded as a polyhedral fan obtained as a common refinement of
hypersurfaces. We then compare it to the set of matrices of Kapranov
rank $\leq 3$, regarded as a subfan of the Gr\"obner fan of
$J^{55}_3$. For the computations, we apply the software
Gfan~\cite{gfan}. We have several techniques for drastically reducing
the computation time, which we describe. We then describe a general
technique for determining whether a set of polynomials forms a
tropical basis. Our analysis shows that the two fans above have the
same support, answering Question~\ref{quest:5x5}.
 
In Section~\ref{sec:5xn}, we prove our main theorem:
\begin{theorem}
\label{thm:main}
  For $n \geq 4$, the $4 \times 4$ minors of a $5 \times n$ matrix form a tropical basis.
\end{theorem}
\begin{corollary}
Let $A\in\R^{d\times n}$ with $d$ or $n\leq 5$. Then the tropical rank of $A$ equals the Kapranov rank of $A$.
\end{corollary}
Our first successful attempt to answer Question~\ref{quest:5x5} without relying on
computer calculations uses a technique which we call ``development by
a column.'' The proof splits into 10-20 cases which need to be treated
separately. We summarize the idea of this proof in Section~\ref{sec:5xn} and refer to the second version of the arXiv paper~\cite{elenasproof}, which addresses a gap in the first version and has been extended to cover the
$5\times n$ case. Then we give a proof of
Theorem~\ref{thm:main} using the technique of stable
intersections and an analysis of types similar to those in~\cite{ardiladevelin}.
With these techniques, we are able to dramatically reduce the number of
cases to consider.

We note that the authors of \cite{DevelinSantosSturmfels2003} were far
from the first to consider notions of rank in the min-plus
setting. Rather, there is a substantial body of literature along these
lines.  See, for instance, the work of Akian, Gaubert, and Guterman
\cite{akiangaubertguterman2}, \cite{akiangaubertguterman}, Izhakian
and Rowen~\cite{izhakianrowen}, and Kim and Roush~\cite{kimroush}; see
\cite[Figure~1]{akiangaubertguterman} for a comparison of ten
different notions of rank.

We finish this introduction with the following conjecture.
\begin{conjecture}
The $(r+1)\times(r+1)$ minors of a $d\times n$
matrix of variables are a tropical basis if and only if $r\leq 2$ or
$r\geq \textup{min}(d,n)-2$.
\end{conjecture}
The next open case of $6\times 6$ matrices is of particular interest because an example of a $6\times 6$ matrix with tropical rank less than Kapranov rank would show that there are nonmatroidal obstructions to the equality of said ranks. Indeed, every matroid on at most six elements is representable over a field of characteristic $0$; see~\cite[Section 3(a)]{blackburn1973}.

We also ask: for which numbers $n$ and $r$ do the $(r+1)\times(r+1)$ minors of a $n\times n$ symmetric matrix form a tropical basis? What about Hankel matrices?

\vspace{0.6cm}
\noindent
{\bf{Acknowledgments:}}
We would like to thank Eva Feichtner and Bernd Sturmfels for helpful conversations, and Bernd for detailed comments on this paper.

\section{Definitions and notation}
\label{sec:def}
We remind the reader of the basic definitions in tropical geometry and
give \cite{maclagansturmfels} and \cite{tropgrass} as references. Let $K$ be the field whose
elements are power series in $t$ with complex coefficients and real
exponents, such that the set of exponents involved in a series is a
well-ordered subset of $\R$. The valuation map $\textup{val}:K^*\rightarrow
\R$ takes a series to the exponent of its lowest order term. 
Denote by
$\textup{val}:(K^*)^N\rightarrow\R^N$ the $N$-fold Cartesian product
of $\textup{val}$. The \emph{tropicalization} of a subvariety $V(I)$
of the torus $(K^*)^N$ defined by an ideal $I\subseteq
K[x_1,\dots,x_N]$ is $\textup{val}(V(I))\subseteq\R^N$.
(With small modifications to our definitions, we expect the results in this paper to hold for any complete algebraically closed non-Archimedean valued field $K$ with the image of the valuation map being dense in $\R$. In particular, the results are independent of the characteristic of the residue field of $K$.)

For
$\omega\in \R^N$, the \emph{$\omega$-degree} of a monomial $cx^a=cx_1^{a_1}\cdots x_N^{a_N}$ is
$\textup{val}(c)+\langle\omega,a\rangle$. The \emph{initial form}
$\textup{in}_\omega(f)\in \C[x_1,\dots,x_N]$ of a polynomial $f\in
K[x_1,\dots,x_N]$ with respect to $\omega$ is the sum of terms of the form
$\gamma t^bx^a$ ($\gamma\in \C$) in $f$ with minimal $\omega$-degree, but with $1$
substituted for $t$. Define the \emph{initial ideal}
$$\textup{in}_\omega(I):=\langle \textup{in}_\omega(f):f\in
I\rangle\subseteq \C[x_1,\dots,x_N].$$
The Fundamental Theorem of Tropical Geometry, variously attributed to Draisma, Kapranov, Speyer-Sturmfels (see~\cite{draisma},\cite{tropgrass}), says that $\textup{val}(V(I))$ equals the \emph{tropical variety} $T(I)$, with
$$T(I):=\{\omega\in\R^N:\textup{in}_\omega(I)\textup{ does not contain
  a monomial}\}.$$

The \emph{Gr\"obner complex} $\Sigma(I)$ of a homogeneous ideal $I$, see \cite{maclagansturmfels}, is the polyhedral complex consisting of all polyhedra
$$C_\omega(I):=\overline{\{\omega'\in\R^N:\textup{in}_\omega(I)=\textup{in}_{\omega'}(I)\}},$$
where $\omega$ runs through $\R^N$, and the closure is taken in the usual Euclidean topology of
$\R^N$. It is clear that the tropical variety $T(I)$ is the support of a subcomplex of $\Sigma(I)$, and we shall not distinguish between $T(I)$ and this subcomplex.

By the \emph{linear span} of
a polyhedron $P\subseteq \R^N$ we mean the $\R$-span of $P-P:=\{p-p':p,p'\in P\}$. The
intersection of the linear spans of all the polyhedra in a complex is called
the \emph{lineality space} of the complex. A complex is invariant
under translation by elements of its lineality space. Since $I$ is
homogeneous, the lineality space of $\Sigma(I)$ contains the $(1,\dots,1)$ vector and it makes sense to
consider $T(I)$ in the \emph{tropical projective torus} ${\TP}^{N-1}:=\R^N/\sim$, where
we mod out by coordinate-wise tropical multiplication by a constant.

If $I$ is a principal ideal $\langle f\rangle$, where
$f=\sum_ic_ix^{a_i}$ with $c_i\in K^*$, the tropical variety is called a \emph{hypersurface}. It
consists of all $\omega\in\R^N$ such that the minimum
\begin{equation}
\label{eq:tropmin}
\bigoplus_i \textup{val}(c_i)\odot \langle \omega,a_i\rangle
\end{equation}
 is
attained at least twice.

In the special case where $I$ is defined by polynomials with
coefficients in $\C$, the complex $\Sigma(I)$ is a fan.  In this paper, we
study two kinds of tropical varieties: those defined by linear ideals in $K[x_1,\dots,x_N]$, which yield polyhedral complexes; and those which are sets of matrices of Kapranov rank at most $r$, and are therefore polyhedral fans (since their ideals are defined over $\C$). In the latter case, we use the terms \emph{Gr\"obner fan} and \emph{Gr\"obner cones} for $\Sigma(I)$ and its cones. 

\begin{definition}
Let $A,B$ be polyhedral fans. The \emph{common refinement} of $A$ and $B$ is the fan
$$A\wedge B:=\{a\cap b:(a,b)\in A\times B\}.$$
\end{definition}

\begin{definition}
Given a set $\F=\{f_1,\dots,f_m\}\subseteq K[x_1,\dots,x_N]$, its \emph{tropical prevariety} is the intersection
$$\bigcap_i T(\langle f_i\rangle).$$
The set $\F$
is a \emph{tropical basis} if its prevariety equals $T(\langle f_1,\dots,f_m\rangle)$. If each $T(\langle f_i\rangle)$ is a fan, then the prevariety can be regarded as their common refinement and hence is a fan.
\end{definition}

We can now give precise definitions of rank.
\begin{definition}
Let $\F_r^{dn}\subseteq K[x_1,\dots,x_{dn}]$ be the set of
$(r+1)\times(r+1)$ minors of the $d\times n$ matrix $\{x_{ij}\}$. Let $J_r^{dn}=\langle f:f\in\F_r^{dn}\rangle$, and $A\in
  \R^{d\times n}$.
\begin{itemize}
\item $A$ has \emph{tropical rank} at most $r$ if 
$A\in\bigcap_{f\in \F_r^{dn}}T(\langle f\rangle)$.
\item $A$ has \emph{Kapranov rank} at most $r$ if $A\in T(J_r^{dn})$.
\end{itemize}
\end{definition}
Equivalently, by the Fundamental Theorem, a matrix $A$ has Kapranov
rank at most $r$ if it has a \emph{lift} $\tilde{A}$ over $K$ of rank
at most $r$. By a lift we mean a matrix $\tilde{A}$ such that
$\textup{val}(\tilde{A})=A$.
\begin{example}
Let $f\in K[x_{11},\dots,x_{33}]$ be the $3\times 3$ determinant, 
$$A=\left(\begin{array}{rrr}
0&1&2\\
1&1&1\\
0&1&1\\
\end{array}\right)~~~~~~~\textup{ and }~~~~~~~~
\tilde{A}=\left(\begin{array}{rrr}
  1&t&t^2\\ 2t&3t&5t\\ 1+2t&4t&5t+t^2\\
\end{array}\right).$$
The tropical hypersurface $T(\langle f\rangle)$ contains $A$
since (\ref{eq:tropmin}) attains its minimum thrice. Hence, $A$ has tropical rank
$\leq 2$. Equivalently, $\textup{in}_A(f)=x_{11}x_{22}x_{33}-x_{11}x_{23}x_{32}+x_{12}x_{23}x_{31}$ is not a monomial. The tropical rank is
not $\leq 1$ since
$\textup{in}_A(x_{11}x_{22}-x_{12}x_{21})=x_{11}x_{22}$.  To argue about the Kapranov rank, we consider the lift $\tilde A\in
K^{3\times 3}$ above. The classical rank of $\tilde A$ is $2$. By the fundamental theorem, or since $\textup{in}_A(J^{33}_2)$ is monomial-free, $A$ has Kapranov rank at most $2$. By (\ref{eq1}), it is equal to $2$.
\end{example}

\section{The $5\times 5$ case}
\label{sec:5x5}
In this section, we compute the prevariety of $5\times 5$
matrices of tropical rank at most $3$ as the common refinement of the 25 hypersurfaces in
$\R^{5\times 5}$ defined by the $4\times 4$ minors of the following matrix:
\begin{footnotesize}
$$\left(\begin{array}{rrrrr}
x_{11} &  x_{12} & x_{13} & x_{14} & x_{15} \\
x_{21} &  x_{22} & x_{23} & x_{24} & x_{25} \\
x_{31} &  x_{32} & x_{33} & x_{34} & x_{35} \\
x_{41} &  x_{42} & x_{43} & x_{44} & x_{45} \\
x_{51} &  x_{52} & x_{53} & x_{54} & x_{55} \\
\end{array}\right).$$
\end{footnotesize}We will then compare it to the tropical variety $T(J_3^{55})$.

Most of our results in this section are based on computer calculations. We explain how to
reproduce these results on our webpage\\
{\footnotesize\url{http://www.math.tu-berlin.de/~jensen/software/gfan/examples/4x4of5x5}}\\
which also contains the complete output of our computations. We have used the software
Gfan~\cite{gfan} and the linear programming libraries
cddlib~\cite{cdd} and SoPlex~\cite{wunderling} with LP-certificates
being verified in exact arithmetic.

First, we compute the hypersurfaces of the
$4\times 4$ minors. Each hypersurface $H_{i}$ is the codimension-one
skeleton of the inner normal fan of the Newton polytope $\textup{New}(g)$ of a $4\times 4$
minor $g\in\F_3^{55}$. The normal fan is denoted $\textup{NF}(\textup{New}(g))$. The \emph{f-vector} records the number of cones of each
dimension, starting with the lineality space. We have
$$\textup{f-vector}(H_{i})=(1, 16, 120, 528, 1392, 2176, 1968, 978, 240),$$
meaning that each $H_{i}$ has $240$ cones of dimension $25-1$, and one, the lineality space, of dimension $16$. We shall refer to the $17$-dimensional cones as \emph{rays}, since they have dimension one modulo the lineality space.

The Newton polytope of a $4\times 4$ minor is a Birkhoff polytope which has
a symmetry group of order $4!\cdot 4!\cdot 2=1152$, namely permutations of coordinates
according to row-interchange, column-interchange and transposition of
the matrix. The symmetries are also seen in the hypersurface: there are three orbits of maximal cones, consisting of 72, 72, and 96 cones, respectively.

While, in theory, the $24$-fold common refinement of the hypersurfaces can be computed in $240^{25}$
iterations by running through an enumeration tree with $25$ levels,
each with $240$ choices, care must be taken for the computation to
finish. It is essential to cut branches off of the enumeration
tree. Gfan does this by writing the support of each hypersurface
as a \emph{disjoint} union of half open cones. Now, when we reach a
node in the enumeration tree, it may happen that the intersection of
the chosen half open cones along the path from the root is empty. In this case,
we may ignore the subtree of the node, see~\cite[Section~7.2]{phdthesis}.


Another trick is to exploit the $5!\cdot 5!\cdot 2=28800$ order symmetry of the common
refinement (interchanging rows and columns and transposing) by
restricting the computation to a fundamental domain of the group
action on $\R^{25}$.
In general, let $G\subseteq S_N$ be a subgroup acting on $\R^N$ by permuting coordinates. Choose a total ordering $\preceq$ on $\R^N$ satisfying
$$u\preceq v\Rightarrow (u+w\preceq v+w\textup{ and } su\preceq sv) \textup{ for all } u,v,w\in\R^N\textup{ and }s\in\R_{>0}.$$
Consider the fundamental domain of the group action
$$\{\omega\in\R^N:\forall \sigma\in G:\omega\preceq
\sigma(\omega)\}=\bigcap_{\sigma\in
  G}\{\omega\in\R^N:\omega\preceq\sigma(\omega)\}.$$ It is easy to see
that the sets in the intersection are convex. With a separation
argument, one can show that their closures are closed half spaces.
Hence the intersection of their closures is a closed polyhedral cone whose orbit covers all
of $\R^N$. In our case, intersecting this cone
with the hypersurface fans (after modding out with a $9$-dimensional lineality space), we get new fans with fewer than 240
maximal cones. Computing the common refinement of the new fans, we get
a set of $1505320$ half open cones of various dimensions whose orbits under the symmetry group cover the prevariety $\bigcap_i H_i$. The fan structure can be different from that of $\bigwedge_i H_i$ since we refined with the fundamental domain. However, the cones of $\bigwedge_i H_i$ are easy to reconstruct: simply pick a relative interior point $\omega$ of each computed cone and compute the cone in $\bigwedge_{f\in\F_3^{55}} \textup{NF}(\textup{New}(f))$ containing it in its relative interior. Although many cone orbits are computed more than once,
this symmetry trick reduces the computation time considerably. The entire
computation took two weeks on a four processor computer to finish, and we obtain the following result:

\begin{proposition}
\label{prop:prevariety}
The set of $5\times 5$ matrices of tropical rank at most $3$, equipped with its common refinement structure, is a non-simplicial, pure $21$-dimensional polyhedral fan in $\R^{25}$ with f-vector:
$$(1, 1450, 28450, 257300, 1418450, 5309320, 14197000, 27724300,$$
$$ 39608950, 40645950, 28590990, 12424200, 2521800).$$
Its lineality space $L$ has dimension $9$. As a spherical complex on the $15$-sphere in $\R^{25}/L$ it has Euler characteristic $-3120$.
\end{proposition}

The maximal cones of the common refinement come in $162$ orbits, while the 1450 rays come in 5 orbits with representatives listed below: 

\begin{footnotesize}
\begin{equation}
\label{eq:vec1}
\pm
\left(\begin{array}{rrrrr}
16 & -4 & -4 & -4 & -4 \\
-4 & 1 & 1 & 1 & 1 \\
-4 & 1 & 1 & 1 & 1 \\
-4 & 1 & 1 & 1 & 1 \\
-4 & 1 & 1 & 1 & 1 \\
\end{array}\right),
~
\left(\begin{array}{rrrrr}
 -4 & -4 & -4 & 6 & 6\\
 -4 & -4 & -4 & 6 & 6\\
 -4 & -4 & -4 & 6 & 6\\
6 & 6 & 6 & -9 & -9\\
6 & 6 & 6 & -9 & -9\\
\end{array}\right),
\end{equation}

$$\left(\begin{array}{rrrrr}
 -3 & -3 & -3 & 7 & 2\\
 -3 & -3 & -3 & 7 & 2\\
 -3 & -3 & -3 & 7 & 2\\
7 & 7 & 7 & -8 & -13\\
2 & 2 & 2 & -13 & 7\\
\end{array}\right),
~
\left(\begin{array}{rrrrr}
 3 & 3 & -7 & -7 & 8\\
 3 & 3 & -7 & -7 & 8\\
 -7 & -7 & 8 & 8 & -2\\
 -7 & -7 & 8 & 8 & -2\\
8 & 8 & -2 & -2 & -12\\
\end{array}\right).$$
\end{footnotesize}The first three matrices have tropical and Kapranov rank $2$, while the ranks are 3 for the last two matrices.

Next, we consider the set of $5\times 5$ matrices of Kapranov rank at most $3$, that is, the tropical
variety $T(J_3^{55})$. It inherits an
underlying polyhedral fan structure from the Gr\"obner fan $\Sigma(J_3^{55})$. Since $J_3^{55}$ is prime, by \cite[Theorem~14]{ctv}, $T(J_3^{55})$ is connected in codimension one  and can be traversed by \cite[Algorithm~8]{ctv}. A one hour Gfan computation, taking advantage of symmetry, gives:
\begin{proposition}
\label{prop:variety}
The set of $5\times 5$ matrices of Kapranov rank at most $3$, considered as a subfan of the Gr\"obner fan of $J^{55}_3$, is a simplicial, pure $21$-dimensional polyhedral fan in $\R^{25}$ with f-vector:
$$(1, 3250, 53650, 421750, 2076700, 7112320, 17790400, 33156700,$$
$$ 46002550, 46497750, 32556390, 14179200, 2894400).$$
Its lineality space $L$ has dimension $9$. As a spherical complex on the $15$-sphere in $\R^{25}/L$ it has Euler characteristic $-3120$. Each maximal cone has tropical multiplicity one.
\end{proposition}
Along  with the 1450 rays of the prevariety, the tropical variety has 1800 additional rays in one orbit with orbit representative
\begin{footnotesize}
\begin{equation}
\label{eq:vec2}
\left(\begin{array}{rrrrr}
 -2 &  -2 & 0 & 2 & 2\\
 -2 &  -2 & 0 & 2 & 2\\
 0 &  0 &2 &  -1 &  -1\\
2 & 2 & -1 & 1 &  -4\\
2 & 2 & -1 &  -4& 1 \\
\end{array}\right).
\end{equation}
\end{footnotesize}Each has tropical and Kapranov rank $3$. The tropical convex hull of the columns of the matrix (\ref{eq:vec2}), shown in Figure~\ref{fig:convexhull}, is contained in a tropicalized $2$ plane in $\TP^4$. The number of orbits of maximal cones in the prevariety is 175. Later in this section we will explain how the number of orbits changed from 162 to 175.
We do not know whether the tropical variety is shellable; the work of Markwig and Yu shows that the corresponding variety of matrices of rank at most two is in fact shellable with a suitable fan structure\cite{markwigyu}.  We would expect the homology of our variety to be concentrated in the top dimension in view of the work of Hacking \cite{hacking}.
\begin{figure}[htpd]
\begin{center}
\hspace{-0.5cm}\scalebox{1}{\input{convexhull.pstex_t}}
\end{center}
\caption{The convex hull in $\TP^4$ of the columns of the matrix (\ref{eq:vec2}). This is a polyhedral complex with f-vector $(7,9,3)$. Its support is not convex in the classical sense.}
\label{fig:convexhull}
\end{figure}
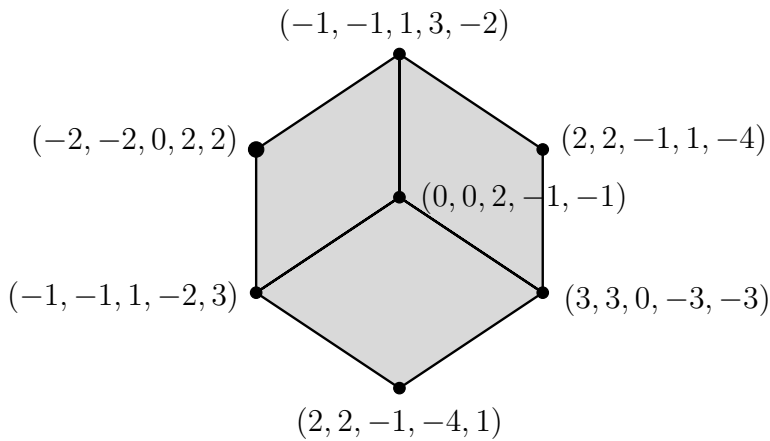

We wish to show that the fans of Proposition~\ref{prop:prevariety} and ~\ref{prop:variety} have the same support. We have seen that they have the same dimension, lineality space, and Euler characteristic; furthermore, random points from the support of one fan can
be checked for containment in the support of the other. In addition,
both fans are tropically balanced (with weight $1$) and connected in codimension 1. By further
investigation of the links of ridges of the two fans, it is possible
to come up with an ad hoc argument that the fans must have the same
support. However, a general method for checking that a tropical
prevariety equals a tropical variety is more appropriate.

We now describe such a method.
Our idea is to
compute the Gr\"obner fan of the ideal inside each maximal cone of the
prevariety. In other words, we compute the common refinement of the prevariety with the Gr\"obner fan. Such a restricted Gr\"obner fan
computation is more complicated to implement than a usual Gr\"obner fan
computation. We refer to \cite{symmetricfans} for a concrete algorithmic strategy.
Since the Gr\"obner fan is complete, the resulting fan is just a refinement of the prevariety. For each cone in the refinement, we now check if it is contained in
the tropical variety by checking that its initial ideal is monomial-free.

Another Gfan computation reveals that in our case the common refinement of the prevariety and the Gr\"obner fan happens to equal the tropical variety, so we do not have to compute any initial ideals:
\begin{proposition}
\label{prop:refinement}
The common refinement of the prevariety defined by the $4\times 4$ minors of a $5\times 5$ matrix and $\Sigma(J_3^{55})$ equals $T(J_3^{55})$ as a fan. That is, the tropical variety is itself a refinement of the prevariety.
\end{proposition}
\begin{corollary}
The $4 \times 4$ minors of a $5 \times 5$ matrix form a tropical basis.
\end{corollary}
\begin{proof}
The corollary follows from the proposition since refining with $\Sigma(J_3^{55})$ does not change the support of the prevariety.
\end{proof}

We note that the tropical basis test described above can be combined with the constructive proof of the existence of tropical
bases~\cite[Theorem~11]{ctv} to give an extended version of
\cite[Algorithm~5]{ctv} for computing tropical bases of general ideals, not just ideals defining curves:

\begin{algorithm}[Tropical basis]$ $

\noindent
{\bf Input:} A finite set of generators $\F\subseteq \C[x_1,\dots,x_N]$ of an ideal $I:=\langle\F\rangle\subseteq K[x_1,\dots,x_N]$.\\
{\bf Output:} A tropical basis of $I$.
\begin{enumerate}
\item Compute the common refinement $A:=\bigwedge_{f\in \F}T(\langle f\rangle)$.
\item For every cone $B\in A$, compute the common refinement
  $D:=\textup{faces}(B)\wedge \Sigma(I)$, by computing all Gr\"obner
  cones inside $B$.
\begin{itemize}
\item For every cone in $E\in D$, choose a relative interior point $\omega\in E$ and check if $\textup{in}_\omega(I)$ contains a monomial. If so, find a ``witness'' $f$ using the proof of \cite[Theorem~11]{ctv}, add it to $\F$, and restart the algorithm.
\end{itemize}
\item Output $\F$.
\end{enumerate}
Here $\textup{faces}(B)$ denotes the fan of all faces of $B$, including $B$ itself.
\end{algorithm}
\begin{proof}
The important property of the ``witness'' $f$ is that $T(\langle
f\rangle)$ does not intersect the relative interior of
$C_\omega(I)$. Therefore, adding $f$ to $\F$ excludes $C_\omega(I)$
from $A$ when the algorithm is restarted.  Termination of the
algorithm follows from the finiteness of the Gr\"obner fan.
\end{proof}


\begin{remark}
The fact that the variety refines the prevariety as in Proposition~\ref{prop:refinement} is not a coincidence. In fact, a Gr\"obner basis argument, which we omit, shows that
$\Sigma(J_{r}^{dn})$ refines $\bigwedge_{f\in\F_r^{dn}}\textup{NF}(\textup{New}(f))$ for any $d,n,r\in \N$.
\end{remark}

In the rest of this section, we explain how the 162 orbits of maximal cones in the prevariety got subdivided into 175 orbits when refining with the Gr\"obner fan.

Let $D$ denote the cone of the prevariety containing the vector
(\ref{eq:vec2}) in its relative interior. The cone $D$ is simplicial
of dimension $19$ and is generated by $10$ rays from the positive
version of the first orbit listed in (\ref{eq:vec1}). When refining,
$D$ gets subdivided into $10$ simplicial cones, spanned by (\ref{eq:vec2}) and each of the $\binom{10}{9}$ possible choices of $9$ remaining rays. The initial ideals of these $10$ cones are equal up to degree $5$. After saturating with $x_{11}\cdots x_{55}$ they all equal the saturation of the initial ideal of (\ref{eq:vec2}):

\vspace{0.1cm}
\begin{footnotesize}
\noindent
\center{$\langle
x_{22}x_{31}-x_{21}x_{32},
x_{13}x_{22}-x_{12}x_{23},
x_{13}x_{21}-x_{11}x_{23},
x_{12}x_{31}-x_{11}x_{32},
x_{12}x_{21}-x_{11}x_{22},$}
{\center $x_{22}x_{34}x_{45}x_{53}+x_{23}x_{32}x_{45}x_{54}+x_{22}x_{35}x_{43}x_{54}-x_{22}x_{33}x_{45}x_{54},$}
{\center $x_{21}x_{34}x_{45}x_{53}+x_{23}x_{31}x_{45}x_{54}+x_{21}x_{35}x_{43}x_{54}-x_{21}x_{33}x_{45}x_{54},$}
{\center $x_{12}x_{34}x_{45}x_{53}+x_{13}x_{32}x_{45}x_{54}+x_{12}x_{35}x_{43}x_{54}-x_{12}x_{33}x_{45}x_{54},$}
{\center $x_{11}x_{34}x_{45}x_{53}+x_{13}x_{31}x_{45}x_{54}+x_{11}x_{35}x_{43}x_{54}-x_{11}x_{33}x_{45}x_{54}\rangle$.}\\
\end{footnotesize}
\vspace{0.1cm}

There are six maximal cones
in the prevariety containing $D$. Denote them by $B_1,\dots,B_6$. They are all
simplicial with twelve generators of the first, positive type of
(\ref{eq:vec1}). The six cones get subdivided into ten cones
each. This accounts for a difference of $1800(60-6)$ in the
number of maximal cones. Counting orbits is more complicated. $B_1$
and $B_2$ belong to the same orbit and so do $B_5$ and $B_6$. Each of
these orbits splits into three new ones under refinement. The orbits of
$B_3$ and $B_4$ both get split in two orbits. This accounts for an
increase of $6$ in the number of orbits.

We now consider the
non-simplicial maximal cones of the prevariety. There is a total of $275400$
such cones in 16 orbits and each splits into two cones. This accounts for the
remaining difference in the number of maximal cones:
$$2894400-2521800=1800(60-6)+275400(2-1).$$ When a non-simplicial cone
splits into two simplicial cones, these can be in the same orbit or
in different orbits. In 7 of the 16 cases, they are in different
orbits. This accounts for the remaining increase of 7 in the number of
orbits after refinement.

\section{The $5\times n$ case}
\label{sec:5xn}
The goal of this section is to prove our main theorem, Theorem~\ref{thm:main}.
As mentioned in the introduction, we have two proof strategies.

We only briefly describe the first strategy here and refer to \cite{elenasproof} for the complete proof.
If $A$ is a matrix, we let $A_{\cdot,\hat{i}}$
be the matrix obtained from  $A$ by removing the $i$-th column $A^{(i)}$.
To prove Theorem~\ref{thm:main},
we  prove that, apart from a few cases which need special treatment,
 every $5 \times n$ matrix $A$ with tropical rank $3$
can be developed by one of its columns, that is, there exists
$i \in \{1,...,n\}$ such that  for any
lift $F$ of $A_{\cdot ,\hat{i}} $, we can find coefficients
in $K$ such that the linear combination of the columns of $F$
with these coefficients is a lift of the column $A^{(i)}$.
Obviously, this linear combination and $F$ give a lift $\tilde{A}$ of $A$ and if
we choose $F$ of rank at most $3$, then the lift $\tilde{A}$ has rank at most $3$.
Theorem~\ref{thm:main} follows by induction on $n$ with $5\times 4$ being the base case.

We now explain the idea of the second proof, which we present in this section. Given a $5\times n$ matrix of tropical rank at most 3, we will produce five tropical hyperplanes, each containing the $n$ column vectors of $A$. Then, we will argue that some pair of these hyperplanes must contain these $n$ points in their stable intersection and conclude that the Kapranov rank is at most 3. The central argument is an analysis of the possible combinatorial types of the point-hyperplane incidences. We begin with some definitions.
\begin{definition}
\label{def:type}
  Let $H = h_1 \odot x_1 \oplus \cdots \oplus h_d \odot x_d$ be a tropical hyperplane in $\TP^{d-1}$, and let $w = (w_1,\dots,w_d)$ be a point in $\TP^{d-1}$.
Then the \emph{type} of $w$ with respect to $H$, denoted $\type_Hw$, is the subset of $[d]$ of those indices at which the minimum of $h_1 \odot w_1, \dots, h_d \odot w_d$ is attained. Thus, $\left|\type_H w\right| \geq 2$ if and only if $w \in H$.
\end{definition}
Note that our definition of type is similar to the definition by Ardila and Develin in \cite{ardiladevelin} which gives rise to tropical oriented matroids. The only difference is that our types are taken with respect to a single hyperplane instead of a hyperplane arrangement.

Types have a natural geometric interpretation, as follows. A tropical
hyperplane $H$ divides $\TP^{d-1}$ into $d$ \emph{sectors}: the $i$th
closed sector consists of those points $w$ for which the minimum
when $H$ is evaluated at $w$ is attained at coordinate $i$. The type
of a point records precisely in which closed sectors it lies.

Recall that a \emph{tropicalized linear space} is the tropicalization of a classical linear variety in $K[x_1,\dots,x_d]$. If the linear space is a classical hyperplane, then we just call its tropicalization a \emph{tropical hyperplane} as in Definition~\ref{def:type}. The \emph{stable intersection} of two tropical linear spaces $L$ and $L'$ is
$$\textup{lim}_{v\rightarrow 0} L\cap(L'+v)$$
and is itself a tropicalized linear space, see \cite[Proposition~4.4.1, Theorem~4.4.6]{speyerthesis} or \cite[Proposition~3.1, Theorem~3.6]{speyertropicallinearspaces}. To clarify, a point $w$ lies in the stable intersection of $L$ and $L'$ if and only if for all $\varepsilon > 0$, there exists $\delta > 0$ such that for each $v \in \R^n$ with $\|v\|_\infty < \delta$, there exists $\tilde{w} \in L \cap (L' + v)$ with $\|\tilde{w} - w\|_\infty < \varepsilon$. (We use the $L^\infty$ norm in our definition for ease of exposition; it is equivalent to using the $L^2$ norm by a standard argument in analysis.)

\begin{proposition}
Let $H$, $H'$ be hyperplanes in $\TP^{d-1}$, and let $w \in \TP^{d-1}$ be a point lying on both $H$ and $H'$. Then $w$ does not lie in their stable intersection precisely when $\type_Hw=\type_{H'}w$ and they are a set of size two.
\end{proposition}
\begin{proof}

Given $w$, $H$, and $H'$, let $\Delta$ be the difference between the minimum and the second smallest number when $H$ is evaluated at $w$, or $\infty$ if only one value occurs. Define $\Delta'$ with respect to $H'$ similarly. Also, write
$$  H = a_1 \odot x_1 \oplus \cdots \oplus a_d \odot x_d,\\
  H' = b_1 \odot x_1 \oplus \cdots \oplus b_d \odot x_d,$$
for real numbers $a_i$, $b_i$.

Suppose $\type_H w \neq \type_{H'} w$ or $|\type_H w| \geq 3$ or $|\type_{H'} w| \geq 3$. Permuting if necessary, we may assume that $1, 2 \in \type_H w$ and $3 \in \type_{H'} w$. Now, given $\varepsilon > 0$, let $\delta = \frac{1}{2} \min\{\varepsilon,\Delta,\Delta'\}$. Let $v \in \R^n$ satisfy $\|v\|_\infty < \delta$. We wish to find a point $\tilde{w} \in H \cap (H' + v)$ such that $\|\tilde{w} - w \|_\infty < \varepsilon$.

If $w \in H'+v$ then we may choose $\tilde{w} = w$ and we are done, so assume instead that the minimum when $H' + v$ is evaluated at $w$ is achieved uniquely, say at coordinate $i$. Furthermore, since $\|v\|_\infty < \frac{1}{2}\Delta'$, the fact that $i \in \type_{H'+v} w$ implies $i \in \type_{H'} w$. We have two cases.

\emph{Case 1.} $i \in \{1,2\}$ and $\type_H w = \{1,2\}$. Let
\[ t = \min_{3 \leq j \leq d} (b_j - v_j + w_j) - (b_i-v_i+w_i), \] and let
$\tilde{w} = (w_1+t,w_2+t,w_3,\dots,w_d)$. Now, we claim that $t < 2\delta$:
\begin{align*}
  t & \leq (b_3 - v_3 + w_3) - (b_i - v_i + w_i) \\
  & \leq |b_3 - v_3 + w_3 - (b_3 + w_3)| 
     + |(b_3 + w_3) - (b_i + w_i)|\\
  & \quad  + |(b_i + w_i) - (b_i - v_i + w_i)| \\
  & \leq |v_3| + |v_i| < 2\delta,
\end{align*}
where the fact that $b_3 + w_3 = b_i + w_i$ follows from the fact that $\{i,3\} \subseteq \type_{H'} w$. Thus, $\|\tilde{w}-w\|_\infty = t < 2\delta \leq \varepsilon$. Also, $\tilde{w}$ lies on $H$ since $\type_H w = \{1,2\}$ and $t < 2\delta \leq \Delta$, and $\tilde{w}$ lies on $H'+v$ by construction.

\emph{Case 2.} $i \notin \{1, 2\}$ or $\type_H w$ strictly contains $\{1,2\}$. In either situation, $\type_H w$ has at least 2 elements different from $i$. Now, let 
\[ t = \min_{j \in [d]\setminus \{i\}} (b_j - v_j + w_j) - (b_i-v_i+w_i), \] and let $\tilde{w} = (w_1,\dots,w_i+t,\dots,w_d)$. Now pick some $k \in \type_{H'} w \setminus \{i\}$; this is possible since $|\type_{H'} w | \geq 2$. Then
\begin{align*}
t &\leq b_k - v_k + w_k - (b_i - v_i + w_i)\\
 & \leq |b_k - v_k + w_k - (b_k + w_k)|  + |(b_k + w_k) - (b_i+w_i)| \\
& \quad + |(b_i+w_i)-(b_i-v_i+w_i)| \\
& \leq |v_k|+|v_i| < 2\delta,
\end{align*}
where $i, k \in \type_{H'} w$ implies that $b_k + w_k = b_i + w_i$. So $\|\tilde{w}-w\|_\infty = t < 2\delta \leq \varepsilon$. Also, $\tilde{w}$ lies on $H$ since $\type_H w$ has at least two elements different from $i$, and $\tilde{w}$ lies on $H' + v$ by construction, as desired.

For the converse, suppose that $\type_H w = \type_{H'} w$ is a set of size two; we may assume it is $\{1,2\}$.
Let $P$ be the affine linear span of the face in $H$ containing $w$. This equals the affine linear span of the face in $H'$ containing $w$ since $\textup{type}_Hw=\textup{type}_{H'}w$. Since the faces of $H$ (and $H'$) are closed, and $|\textup{type}_Hw|=2$ (and $|\textup{type}_{H'}w|=2$) implies that $w$ is contained in just one face of $H$ (and $H'$), there exists $\varepsilon>0$ such that
$$H\cap B(w,2\varepsilon)=P\cap B(w,2\varepsilon)=H'\cap B(w,2\varepsilon),$$
where $B(w,\varepsilon)$ is the $\varepsilon$-ball centered at $w$. For any $\delta>0$, pick $v=\textup{min}(\delta/2,\varepsilon)(e_1-e_2)$. Now
$$B(w,\varepsilon)\cap(H'+v)=((B(w,\varepsilon)-v)\cap H')+v\subseteq$$
$$ (B(w,2\varepsilon)\cap H')+v=
(B(w,2\varepsilon)\cap P)+v\subseteq P+v$$
which shows that $B(w,\varepsilon)\cap(H'+v)\cap H\subseteq (P+v) \cap P= \emptyset$.
\end{proof}
\begin{proposition}
\label{prop:stablelift}
Let $H,H'$ be tropical hyperplanes in $\TP^{d-1}$. Then there exists a codimension $2$ linear space $L$ over $K$ whose tropicalization is the stable intersection of $H$ and $H'$.
\end{proposition}
\begin{proof}
By \cite[Proposition~4.5.3]{speyerthesis}, we may lift $H$ and $H'$ generically to classical hyperplanes $\HH$ and $\HH'$ over $K$ such that the tropicalization of $\HH\cap \HH'$ is the stable intersection of $H$ and $H'$.
\end{proof}
Now let $W$ be a set of points in $\TP^{d-1}$, and let $i\in\{1,\dots,d\}$. We say that a hyperplane $H$ that contains each point in $W$ is an \emph{$i$-coordinate hyperplane}
for $W$ if $\type_Hw$ does not contain $i$ for any $w\in W$. That is, no point in $W$ lies in the $i$th closed sector of $H$.

Next, suppose $w$ is a point contained in two hyperplanes $H$ and $H'$ but not in their stable intersection. Then, by Proposition~\ref{prop:stablelift}, $w$ has type $\{a,b\}$ with respect to both $H$ and $H'$, for some $a$ and $b$. Then we say that $w$ is a \emph{witness of type $ab$}  to the nonstable intersection of $H$ and $H'$. In the case that $H$ and $H'$ are $k$- and $l$-coordinate hyperplanes,
  respectively, we note that the sets $\{a,b\}$ and $\{k,l\}$ must be
  disjoint.

We are now ready to state a proposition which will serve as the combinatorial heart of our proof of Theorem~\ref{thm:main}.

\begin{proposition}
\label{prop:3}
  Let $W=\{w_1, \dots, w_n\}$ be a subset of points in $\TP^4$, and for each $i$ with $1 \leq i \leq 5$, let $H_i$ be a hyperplane containing each point in $W$ such that $H_i$ is an $i$-coordinate hyperplane for $W$. Suppose further that for every pair of hyperplanes $H_i$ and $H_j$, the intersection $H_i\cap H_j$ is not the stable intersection of $H_i$ and $H_j$ and some $w_s\in W$ witnesses this nonstable intersection.

Let $i,j,k,l,m$ be distinct elements in $\{1,2,3,4,5\}$. Suppose $H_i$ and $H_j$ have a witness in $W$ of type $kl$. Then any witness in $W$ for $H_i$ and $H_k$ has type $lm$.
\end{proposition}
Proposition~\ref{prop:3} follows from the following two lemmas, whose proofs we postpone to the end of the section.
\begin{lemma}
\label{lem:1}
  Let $W=\{w_1, \dots, w_n\}$, $H_1, \dots, H_5$ be as in the first paragraph of Proposition~\ref{prop:3}.  Let $i,j,k,l$ be distinct elements in $\{1,\dots,5\}$, and assume without loss of generality that $H_i$ and $H_j$ have a witness in $W$ of type $kl$. Then any witness in $W$ for $H_i$ and $H_k$ must be of type containing $l$.
\end{lemma}

\begin{lemma}
\label{lem:2}
  Let $W=\{w_1, \dots, w_n\}$, $H_1, \dots, H_5$ be as in the first paragraph of Proposition~\ref{prop:3}. Let $i,j,k,l$ be distinct elements in $\{1, \dots, 5\}$. Then it is not possible that $H_i$ and $H_k$ have a witness in $W$ of type $jl$ and that $H_i$ and $H_j$ have a witness in $W$ of type $kl$.
\end{lemma}

\begin{proof}[Proof of Proposition~\ref{prop:3}]
$H_i$ and $H_k$ have some witness to their nonstable intersection; it must be of type $jl$, $jm$, or $lm$. But it is not of type $jm$ by Lemma~\ref{lem:1}, and it is not of type $jl$ by Lemma~\ref{lem:2}.
\end{proof}

Before proving the lemmas, we first prove that Proposition~\ref{prop:3} implies the main result.

\begin{proof}[Proof of Theorem~\ref{thm:main}]
Fix $n \geq 4$; let $A$ be a $5 \times n$ real matrix, and let $W=\{w_1, \dots, w_n\}$ be the set of its column vectors. Suppose that the tropical rank of $A$ is $\leq 3$. We wish to show that the Kapranov rank is $\leq 3$, or equivalently, that there exists a 3-dimensional subspace in $K^5$ whose tropicalization contains each point $w_1, \dots, w_n$.

Let $A'$ be the $4 \times n$ matrix obtained by deleting the first row
of $A$. Then the tropical rank of $A'$ is $\leq 3$, so by
Theorem~\ref{old theorem}, the Kapranov rank of $A'$ is $\leq 3$, so
the columns of $A'$ lie on some tropical hyperplane, say
\[ h_{12} \odot x_2 \oplus h_{13} \odot x_3 \oplus h_{14} \odot x_4 \oplus h_{15} \odot x_5. \]
Then, for $N$ sufficiently large,
\[ H_1 := N \odot x_1 \oplus h_{12} \odot x_2 \oplus h_{13} \odot x_3 \oplus h_{14} \odot x_4 \oplus h_{15} \odot x_5 \]
is a hyperplane containing the columns of $A$, indeed a 1-coordinate hyperplane, where none of $w_1, \dots, w_n$ has type containing $1$.

Similarly, we may choose $H_2, H_3, H_4, H_5$ to be 2, 3, 4, 5-coordinate hyperplanes, respectively, for the points $w_1, \dots, w_n$.

We claim that for some $i,j$ with $1 \leq i < j \leq 5$, $H_i$ and $H_j$ contain each $w_1, \dots, w_n$ in their stable intersection. If so, we are done by Proposition~\ref{prop:stablelift}.

Suppose, then, that the claim is not true, so that for every $i,j$ with $1\leq i < j \leq 5$, some point in $W$ witnesses the nonstable intersection of $H_i$ and $H_j$. We derive a contradiction as follows.

By symmetry, we may assume that $H_1$ and $H_2$ have a witness of type $34$. We now apply Proposition~\ref{prop:3} four times to get a contradiction. First, $H_1$ and $H_3$ have a witness in $W$ to their nonstable intersection by assumption; it is of type $45$ by Proposition~\ref{prop:3}. Similarly, $H_1$ and $H_4$ have a witness in $W$ of type $35$. Applying Proposition~\ref{prop:3} to these two facts, we get that any witness in $W$ for $H_1$ and $H_5$ must have type $24$, and similarly, that any witness in $W$ for $H_1$ and $H_5$ must have type $23$. Since $H_1$ and $H_5$ do have a witness in $W$, by assumption, this is a contradiction.
\end{proof}

\begin{proof}[Proof of Lemma~\ref{lem:1}]
By symmetry, assume $i=1, j=2, k=4, l=5$. Suppose $H_1$ and $H_2$ have a witness of type 45, and that $H_1$ and $H_4$ have a witness of type not containing 5 -- that is, we assume that $H_1$ and $H_4$ have a witness of type 23. We wish to derive a contradiction.

For each $s$ with $1 \leq s \leq 5$, write
\[ H_s = h_{s1} \odot x_1 \oplus \cdots \oplus h_{s5} \odot x_5, \text{with } h_{sr}  \in \R. \]
By translating each hyperplane and each point, we may assume that
\[ H_1 = 0 \odot x_1 \oplus \dots \oplus 0 \odot x_5, \]
and for each $s$ with $2 \leq s \leq 5$, we may assume, by tropically scaling the coefficients of $H_s$, that $h_{s1} = 0$. Furthermore, since $H_1$ and $H_2$ have a witness of type 45, and $h_{14} = h_{15}$, it follows that $h_{24} = h_{25}$. Similarly, $h_{42} = h_{43}$. Summarizing, we have
\begin{align*}
H_1 &= 0x_1 \oplus 0x_2 \oplus 0x_3 \oplus 0x_4 \oplus 0x_5,\\
H_2 &= 0x_1 \oplus ex_2 \oplus bx_3 \oplus ax_4 \oplus ax_5,\\
H_4 &= 0x_1 \oplus cx_2 \oplus cx_3 \oplus fx_4 \oplus dx_5,
\end{align*}
with $a,b,c,d,e,f \in \R$.
By symmetry (i.e. switching 2 with 4 and 3 with 5), we may assume $a \leq c$.

Now, we claim $b > a$. Indeed, let $\omega = (\omega_1, \omega_2, \omega_3, \omega_4, \omega_5)$ be a witness of type 23 for $H_1$ and $H_4$. Since $\type_{H_1}(\omega) = 23$, we have that the minimum of $\{\omega_1, \omega_2, \omega_3, \omega_4, \omega_5\}$ is attained twice, in fact, precisely at $\omega_2$ and $\omega_3$. Tropically rescaling, we may assume that $\omega_2 = \omega_3 = 0$ and that $\omega_1, \omega_4, \omega_5 > 0$. Since $\type_{H_4}(\omega) = 23$, we have that $\min(\omega_1, c+\omega_2,c+\omega_3,f+\omega_4,d+\omega_5)$ is attained precisely at $c+\omega_2 = c+\omega_3 = c$, so $\omega_1 > c$. Finally, since $\omega \in H_2$, and $H_2$ is a 2-coordinate hyperplane, we have that $\min(\omega_1,b+\omega_3,a+\omega_4,a+\omega_5)$ is achieved twice. Since $\omega_1 > c \geq a$, $\omega_3 = 0$, $\omega_4, \omega_5 > 0$, this is only possible if $b > a$.

Next, we claim $d > a$. The proof is similar. Let $\chi = (\chi_1, \dots, \chi_5) \in \TP^4$ be a witness of type 45 for $H_1$ and $H_2$. Using that $\type_{H_1} \chi = \type_{H_2} \chi = \{4,5\}$ and tropically rescaling, we have $\chi_1 > a$, $\chi_2 > 0$, $\chi_3 > 0$, $\chi_4 = \chi_5 = 0$. Together with $a\leq c$ this implies $c+\chi_2,c+\chi_3>a$. But $\chi \in H_4$ and $H_4$ is a 4-coordinate hyperplane, so $\min(\chi_1,c+\chi_2,c+\chi_3,d+\chi_5)$ is attained twice, and since $\chi_1, c+\chi_2,c+\chi_3$ are all $>a$ and $\chi_5 = 0$, we have $d > a$.

Now, $H_2$ and $H_4$ have some witness of nonstable intersection, say $\psi = (\psi_1,\psi_2,\psi_3,\psi_4,\psi_5) \in \TP^4$, where $\psi$ is a witness of type 13, 15, or 35. Since $h_{21} = h_{41} = 0$, but $h_{25} = a \neq d = h_{45}$, it is not type 15, so it is of type 13 or 35.

Suppose $\psi$ is of type 35, so $\type_{H_2} \psi = \type_{H_4} \psi = \{3, 5\}$. Then, rescaling, we may assume $\psi_3 = a$, $\psi_5 = b$, $\psi_2 > a$, and $\psi_4 > b$. But we showed $b > a$, so $\min(\psi_2, \psi_3, \psi_4, \psi_5)$ is attained uniquely, contradicting that $\psi \in H_1$ and $H_1$ is a 1-coordinate hyperplane.

So $\psi$ must be a witness of type 13 for $H_2$ and $H_4$. Since $h_{21} = h_{41} = 0$, we have $h_{23} = h_{43}$, that is, $b = c$. Furthermore, using $\type_{H_2} \psi = \type_{H_4} \psi = \{1,3\}$, and rescaling $\psi$, we may assume that $\psi_1 = b$, $\psi_2 > 0$, $\psi_3 = 0$, $\psi_4 > b-a$, $\psi_5 > b-a$. But since $b-a > 0$ and $H_1$ is a 1-coordinate hyperplane, $\psi \notin H_1$, contradiction. This proves Lemma~\ref{lem:1}.
\end{proof}

\begin{proof}[Proof of Lemma~\ref{lem:2}]
By symmetry, we may assume $i=1, j=4, k=3, l=2$, and we suppose for a contradiction that $H_1$ and $H_3$ have a witness, $\omega$, of type 24, and $H_1$ and $H_4$ have a witness, $\chi$, of type 23. We may assume, by translating and rescaling, that
\begin{align*}
H_1 &= 0x_1 \oplus 0x_2 \oplus 0x_3 \oplus 0x_4 \oplus 0x_5, \\
H_3 &= 0x_1 \oplus ax_2 \oplus ex_3 \oplus ax_4 \oplus bx_5, \\
H_4 &= 0x_1 \oplus cx_2 \oplus cx_3 \oplus fx_4 \oplus dx_5,
\end{align*}
for some $a,b,c,d,e,f \in \R$. Then, rescaling, we may assume that $\omega = (\omega_1, 0, \omega_3, 0, \omega_5)$ where $\omega_1 > a$, $\omega_3 > 0$, and $\omega_5 > a-b$. Similarly, we may assume that $\chi = (\chi_1, 0, 0, \chi_4, \chi_5)$, where $\chi_1 > c$, $\chi_4 > 0$, and $\chi_5 > c - d$.

Now, by hypothesis, $H_3$ and $H_4$ have some witness $\psi$ to their nonstable intersection; its type must be 12, 15, or 25, since types containing 3 or 4 may not occur.

Suppose it is type 12. Then $a = c$. That $\omega$ lies on $H_4$ implies that $\min(\omega_1, c, c+\omega_3, d+\omega_5)$ is attained twice; since $\omega_1 > a = c$ and $c+\omega_3 > c$, we have $c = d+\omega_5$. Since $\omega_5 > a-b$, we have $c > d+a-b$, so $b > d$. 
Symmetrically, $\chi \in H_3$ implies $\min(\chi_1,a+\chi_2,a+\chi_4,b+\chi_5)$ is achieved twice; since $\chi_1>c=a$, $a+\chi_2=a$, $a+\chi_4>a$ and $b+\chi_5>b+c-d=a+b-d$, it follows that $d > b$, contradiction.

Next, suppose $\psi$ is a witness for $H_3$ and $H_4$ of type 15. Then $b = d$. Then $\omega \in H_4$ implies that $\min(\omega_1,c,c+\omega_3,d+\omega_5)$ is achieved twice; since $\omega_1 > a$, $\omega_3 > 0$ and $d+\omega_5 > d+a-b=a$, it follows that $c > a$ (otherwise the minimum is achieved uniquely at $c$).
Symmetrically, $\chi \in H_3$ implies $\min(\chi_1,a+\chi_2,a+\chi_4,b+\chi_5)$ is achieved twice; since $\chi_1>c$, $a+\chi_2=a$, $a+\chi_4>a$ and $b+\chi_5>b+c-d=c$, it follows that $a > c$, contradiction.

Finally, suppose $\psi$ is a witness of type 25. Then $a+d = b+c$, and $\omega \in H_4$ implies that $\min(\omega_1,c,c+\omega_3,d+\omega_5)$ is attained twice; since $\omega_1 > a$, $c+\omega_3 > c$, and $d+\omega_5 > d+a-b=c$, we have $c > a$.
Symmetrically, $\chi \in H_3$ implies $\min(\chi_1,a+\chi_2,a+\chi_4,b+\chi_5)$ is achieved twice; since $\chi_1>c$, $a+\chi_2=a$, $a+\chi_4>a$ and $b+\chi_5>b+c-d=a$, it follows that $a > c$. This is a contradiction and proves Lemma~\ref{lem:2}.
\end{proof}

\appendix
\bibliographystyle{hplain}
\bibliography{jensen.bib}
\end{document}

%% file: convexhull.pstex_t
\begin{picture}(0,0)%
\epsfig{file=convexhull.pstex}%
\end{picture}%
\setlength{\unitlength}{3947sp}%
\begingroup\makeatletter\ifx\SetFigFont\undefined%
\gdef\SetFigFont#1#2#3#4#5{%
  \reset@font\fontsize{#1}{#2pt}%
  \fontfamily{#3}\fontseries{#4}\fontshape{#5}%
  \selectfont}%
\fi\endgroup%
\begin{picture}(4570,2685)(3538,-4880)
\put(5349,-4831){\makebox(0,0)[lb]{\smash{{\SetFigFont{12}{14.4}{\rmdefault}{\mddefault}{\updefault}{\color[rgb]{0,0,0}$(2,2,-1,-4,1)$}%
}}}}
\put(7009,-3052){\makebox(0,0)[lb]{\smash{{\SetFigFont{12}{14.4}{\rmdefault}{\mddefault}{\updefault}{\color[rgb]{0,0,0}$(2,2,-1,1,-4)$}%
}}}}
\put(7032,-4060){\makebox(0,0)[lb]{\smash{{\SetFigFont{12}{14.4}{\rmdefault}{\mddefault}{\updefault}{\color[rgb]{0,0,0}$(3,3,0,-3,-3)$}%
}}}}
\put(6131,-3415){\makebox(0,0)[lb]{\smash{{\SetFigFont{12}{14.4}{\rmdefault}{\mddefault}{\updefault}{\color[rgb]{0,0,0}$(0,0,2,-1,-1)$}%
}}}}
\put(5241,-2327){\makebox(0,0)[lb]{\smash{{\SetFigFont{12}{14.4}{\rmdefault}{\mddefault}{\updefault}{\color[rgb]{0,0,0}$(-1,-1,1,3,-2)$}%
}}}}
\put(3686,-3060){\makebox(0,0)[lb]{\smash{{\SetFigFont{12}{14.4}{\rmdefault}{\mddefault}{\updefault}{\color[rgb]{0,0,0}$(-2,-2,0,2,2)$}%
}}}}
\put(3538,-4038){\makebox(0,0)[lb]{\smash{{\SetFigFont{12}{14.4}{\rmdefault}{\mddefault}{\updefault}{\color[rgb]{0,0,0}$(-1,-1,1,-2,3)$}%
}}}}
\end{picture}%